\documentclass[11pt]{amsart}
\theoremstyle{plain}
\newtheorem{thm}{Theorem}[section]
\newtheorem{theorem}[thm]{Theorem}

\newtheorem{lemma}[thm]{Lemma}
\newtheorem{corollary}[thm]{Corollary}
\newtheorem{proposition}[thm]{Proposition}
\theoremstyle{definition}
\newtheorem{remark}[thm]{Remark}

\numberwithin{equation}{section}

\newcommand{\PP}{\ensuremath{\mathbb{P}}}

\newcommand{\AAA}{\ensuremath{\mathbb{A}}}
\newcommand{\CC}{\ensuremath{\mathbb{C}}}

\newcommand{\ZZ}{\ensuremath{\mathbb{Z}}}

\title [Unirationality]{Unirationality of Ueno-Campana's threefold}

\author{Fabrizio Catanese, Keiji Oguiso and Tuyen Trung Truong}

\address{}
\email{}

\address{} \email{}

\subjclass[2010]{}

\thanks{}

\begin{document}

\maketitle

\begin{abstract}
We shall prove that the threefold studied in the paper `` Remarks on an Example of K. Ueno" by F. Campana is unirational. 
This gives an affirmative answer to a question posed in the paper 
above and also in the book by K. Ueno, ``Classification theory of algebraic varieties and compact complex spaces". 
\end{abstract}

\section{Introduction}

Let $k$ be any field of characteristic $\not= 2$ containing a primitive fourth root of unity $\sqrt{-1}$. We shall work over $k$ unless otherwise stated. Let 
$[x : y : z]$ be the homogeneous coordinates of ${\PP}^2$ and let
$$C := (y^2z = x(x^2 -z^2)) \subset {\PP}^2$$
be the harmonic elliptic curve, having an automorphism $g$ of order $4$ defined by
$$g^*(x : y : z) = (-x : \sqrt{-1}y :z)\,\, $$
whose quotient is $\PP^1$.
When $k$ is the complex number field ${\CC}$, we have 
$$(C, g) \simeq 
(E_{\sqrt{-1}}, \sqrt{-1})\,\, ,$$ 
where $E_{\sqrt{-1}} = {\CC}/ ({\ZZ} + \sqrt{-1}{\ZZ})$, the elliptic curve of period $\sqrt{-1}$ and $\sqrt{-1}$ is the automorphism induced by  multiplication by $\sqrt{-1}$ on ${\mathbf C}$. This is because the complex elliptic curve with an automorphism of order $4$ acting on the space of global holomorphic $1$-forms as $\sqrt{-1}$ is unique up to isomorphism. 

Let $(C_j, g_i)$ ($j = 1,2,3$) be  three copies of $(C, g)$. Let 
$$Z = C_1 \times C_2 \times C_3\,\, .$$ 
For simplicity, we denote the automorphism of $Z$ defined by $(g_1, g_2, g_3)$ by the same letter $g$. Then $g$ is an automorphism of $Z$ of order $4$ and the quotient threefold 
$$Y := (C_1 \times C_2 \times C_3 )/ \langle g \rangle$$
has $8$ singular points of type 
$(1,1,1)/4$ and $28$ singular points of type $(1,1,1)/2$. Let $X$ be the blow up of $Y$ at the maximal ideals of these singular points. Then $X$ is a smooth projective threefold defined over $k$. In his paper \cite{Ca12}, F. Campana proved that $X$ is a rationally connected threefold when $k = \CC$. We shall call $X$ the  {\it Ueno-Campana's threefold}. 

In \cite[Question 4]{Ca12}, F. Campana asked if $X$ is rational or unirational (at least over ${\CC}$)? See also \cite[Page 208]{Ue75} for this Question and \cite{OT13} for a relevant example and application to complex dynamics. The aim of this short note is to give an affirmative answer to this question:

\begin{theorem}\label{main1}
Ueno-Campana's threefold $X$ is unirational, i.e., there is a dominant rational map ${\PP}^3 \cdots\rightarrow X$. 
\end{theorem}

We shall show that $X$ is birationally equivalent to the Galois quotient of a conic bundle over ${\PP}^2$ {\it with} a rational section, while $X$ itself is birationally equivalent to a conic bundle over ${\PP}^2$ {\it without} any rational section.

We are still working on the question whether $X$ is a rational variety.

{\bf Aknowledgement.} We would like to express our thanks to Professor De-Qi Zhang for his invitation to Singapore where the initial idea of this note grew 
 up. 

\section{Proof of Theorem (\ref{main1})}

The curves $(C_i, g_i)$ ($i =1,2,3$) are birationally equivalent to $(C_i^0, g_i)$, where $C_i^0$ is the curve in the affine space ${\AAA}^2 = {\rm Spec}\, k[X_i, Y_i]$, and $g_i$ is the automorphism of $C_i^0$, defined by
$$Y_i^2 = X_i(X_i^2 -1)\,\, ,\,\, g_i^*Y_i = \sqrt{-1}Y_i\,\, ,\,\, g_i^*X_i = -X_i\,\, .$$
The affine coordinate ring $k[C_i^0]$ of $C_i^0$ is 
$$k[C_i^0] = k[X_i, Y_i]/(Y_i^2 -X_i(X_i^2 -1)) .\, ,$$ 
WE set  $x_i : = X_i\, {\rm mod}\, (Y_i^2 -X_i(X_i^2 -1)), y_i : = Y_i\, {\rm mod}\, (Y_i^2 -X_i(X_i^2 -1))$. We note that $y_i^2 = x_i(x_i^2 -1)$, $g^*y_i = \sqrt{-1}y_i$, $g^*x_i = -x_i$ in $k[C_i^0]$. 

Then $(Z = C_1 \times C_2 \times C_3, g = (g_1, g_2, g_3))$ is birationally equivalent to the affine threefold 
$$V := C_1^0 \times C_2^0 \times C_3^0$$
with automorphism $(g_1, g_2, g_3)$, which we denote by the same letter $g$, and with affine coordinate ring
$$k[V] = k[C_i^0] \otimes k[C_2^0] \otimes k[C_3^0] \ {\rm generated \ by } \ x_1, x_2, x_3, y_1, y_2, y_3\,\, .$$
The rational function field $k(Z)$ of $Z$ is
$$k(Z) = k(V) = k(x_1, x_2, x_3, y_1, y_2, y_3)\,\, .$$
In both $k[V]$ and $k(Z)$, we have 
$${\rm (I)}\,\, y_i^2 = x_i(x_i^2 -1)\,\, ,$$ 
$${\rm (II)}\,\, g^*y_i = \sqrt{-1}y_i\,\, ,\,\, g^*x_i = -x_i\,\, .$$
Since $X$ is birationally equivalent to $V/\langle g \rangle$, the rational function field $K(X)$ of $X$ is identified with the invariant subfield $k(Z)^g$ of $k(Z)$, 
i.e., 
$$k(X) = k(Z)^g = \{f \in k(Z)\, \vert g^*f = f\}\,\, .$$
Consider the following elements in $k(Z)$:
$${\rm (III)}\,\, b_2 := \frac{x_2}{x_1}\,\, ,\,\, b_3 := \frac{x_3}{x_1}\,\, ,\,\, a_2 := \frac{y_2}{y_1}\,\, ,\,\, a_3 := \frac{y_3}{y_1}\,\, ,$$
$${\rm (IV)}\,\, u_1 := x_1^2\,\, ,\,\, w_1 := y_1^4\,\, ,\,\, \lambda_1 := x_1y_1^2\,\, ,$$
and define the subfield $L$ of $k(Z)$ by
$$L := k(b_2, b_3, a_2, a_3, u_1, w_1, \lambda_1)\,\, .$$
Here we used the fact that $x_1 \not= 0$, $y_1 \not= 0$ in $k(Z)$.

\begin{lemma}\label{lem1}
$k(X) = L$ in $k(Z)$.
\end{lemma}

\begin{proof} By (II) and (III), $b_2$, $b_3$, $a_2$, $a_3$, $u_1$, $w_1$, $\lambda_1$ are $g$-invariant. Hence 
$${\rm (V)}\,\, L \subset k(X) \subset k(V)\,\, .$$
Note that $k(Z) = L(y_1)$. This is because 
$$x_1 = \frac{\lambda_1}{y_1^2}\,\, ,\,\, x_2 = b_2x_1, x_3 = b_3x_1\,\, ,\,\, 
y_2 = a_2y_1\,\, ,\,\, y_3 = a_3y_1\,\, ,$$
by (III) and (IV). Since $y_1^4 = w_1$ and $w_1 \in k(Z)$, it follows that 
$${\rm (VI)}\,\, [k(Z) : L] \le 4\,\, ,$$
where $[k(Z) : L]$ is the degree of the field extension $L \subset k(Z)$, i.e., the dimension of $k(Z)$ being naturally regarded as the vector space over $L$. 

On the other hand, the group $\langle g \rangle \subset {\rm Gal}(k(Z)/k(X))$ 
is of order $4$. Thus, by the fundamental theorem of Galois theory, we have that 
$${\rm (VII)}\,\, [k(Z) : k(X)] = [K(Z) : k(Z)^g] =\, {\rm ord}\,(g) 
= 4\,\, .$$
The result now follows from (V), (VI), (VII). Indeed, by (V), we have 
$$[k(Z) : L] = [k(Z) : k(X)][k(X) :L]\,\, .$$
On the other hand, $[k(Z) : L ] \le 4$ by (VI), and $[k(X):L]\geq 1$. Hence $[k(X) : L ] = 1$ by (VII). 
This means that $L = k(X)$ in $k(Z)$, as claimed.
\end{proof}

\begin{lemma}\label{lem2}
$L = k(u_1, b_2, b_3, a_2, a_3)$ in $k(Z)$.
\end{lemma}

\begin{proof} Since $u_1, b_2, b_3, a_2, a_3 \in L$, it follows that $k(u_1, b_2, b_3, a_2, a_3) \subset L$. 
Let us show $L \subset k(u_1, b_2, b_3, a_2, a_3)$. For this, it suffices to show that $w_1, \lambda_1 \in k(u_1, b_2, b_3, a_2, a_3)$. 

Recall by (I), $y_1^2 = x_1(x_1^2 -1)$, Hence taking the square and using (VI), we obtain that
$${\rm (VIII)}\,\, w_1 = y_1^4 = x_1^2(x_1^2 -1)^2 = u_1(u_1 -1)^2\,\, .$$
Hence $w_1 \in k(u_1, b_2, b_3, a_2, a_3)$. From $y_1^2 = x_1(x_1^2 -1)$ again, we have that
$${\rm (IX)}\,\, \lambda_1 = x_1y_1^2 = x_1^2(x_1^2 -1) = u_1(u_1 -1)\,\, .$$
Hence $\lambda_1 \in k(u_1, b_2, b_3, a_2, a_3)$ as well.
\end{proof}

\begin{lemma}\label{lem3} Let $j = 2, 3$. Then, 
$a_j^2 - b_ j \not= 0$ in both $k(Z)$ and $k(X)$. 
\end{lemma}

\begin{proof} By using (I), we obtain that 
$${\rm (X)} \ a_j^2 - b_ j = \frac{y_j^2}{y_1^2} - \frac{x_j}{x_1} = \frac{x_j(x_j^2-1)}{x_1(x_1^2-1)} - \frac{x_j}{x_1} = \frac{x_j}{x_1}(\frac{x_j^2-1}{x_1^2 -1} -1)\,\, ,$$
in $k(Z)$. Recall that $x_i \not= 0$ in $k(Z)$. Thus, if $a_j^2 - b_j = 0$ in $k(Z)$, then we would have $(x_j^2-1)/(x_1^2 -1) = 1$ in $K(Z) = k(V)$ 
from the equality above, and 
therefore, $x_j = \pm x_1$ in 
$k[V]$. However, this contradicts to the fact that $x_1$ is identically $0$ on 
the set of $\overline{k}$-valued points $(\{0\} \times C_2 \times C_3)(\overline{k})$ but $\pm x_j$ ($j =2, 3$) are not identically $0$ on it. This contradiction implies that $a_j^2 - b_j \not= 0$ in 
$k(Z)$. Since $a_j^2 - b_j \in k(Z)^g = k(X)$ and $k(X)$ is a subfield of 
$k(Z)$, it follows that $a_j^2 - b_j \not= 0$ in $k(X)$ as well. 
\end{proof}

\begin{proposition}\label{field}
$k(X) = L = k(b_2, b_3, a_2, a_3)$ in $k(Z)$. More precisely, in $k(Z)$, 
we have
$${\rm (XI)}\,\, u_1 = \frac{a_2^2 - b_2}{a_2^2 - b_2^3} = \frac{a_3^2 - b_3}{a_3^2 - b_3^3}\,\, .$$
\end{proposition}

\begin{proof} By Lemma (\ref{lem1}, \ref{lem2}), it suffices to show the equality (X) in $k(Z)$. 
Observe that, for $j=2,3$:

$$   \ y_j^2 = x_j (x_j^2 -1) \Leftrightarrow y_1^2 a_j^2 = x_1 b_j (x_1^2 b_j^2 -1)$$ 
hence multiplication by $x_1$ yields 
 $$  \  x_1^2  b_j (x_1^2 b_j^2 -1) = x_1 y_1^2 a_j^2 = x_1^2 (x_1^2 -1) a_j^2 , $$
 and dividing by $x_1^2$ and observing that $u_1 = x_1^2$ we obtain
 
 $$ b_j (u_1 b_j^2 -1) = (u_1 - 1) a_j^2$$
 i.e.,
 
 $$   (**) \ u_1 (a_j^2 - b_j^3) = a_j^2  - b_j.$$

Using the previous lemma we obtain $(a_j^2 - b_j^3) \neq 0$, so we can divide and obtain  (XI).

\end{proof}

\begin{proposition}\label{hypersurface}
$X$ is birationally equivalent to the affine hypersurface $H$ in ${\AAA}^4 = {\rm Spec}\, k[a, b, \alpha, \beta]$, defined by
$$(a^2 - b)(\alpha^2- \beta^3) = (\alpha^2 - \beta)(a^2 - b^3)\,\, ,$$
or equivalently defined by
$$a^2\beta(1-\beta^2) = \alpha^2b(1-b^2) +b\beta(b^2 -\beta^2)\,\, ,$$
\end{proposition}

\begin{proof} By Lemma (\ref{lem1}) and Proposition (\ref{field}), $k(X) = k(a_2, a_3, b_2, b_3)$ in $k(Z)$, with a relation 
$${\rm (XII)}\,\, (a_2^2 -b_2)(a_3^2 -b_3^3) = (a_3^2 -b_3)(a_2^2 -b_2^3)
\,\, .$$
Expanding both sides and subtracting then the common term $a_2^2a_3^2$, 
we obtain
$$-a_2^2b_3^3 -b_2a_3^2 + b_2b_3^3 = -a_3^2b_2^3 -b_3a_2^2 + b_3b_2^3\,\,.$$
Solving this relation in terms of $a_2$, we obtain that
$${\rm (XIII)}\,\, a_2^2b_3(1 - b_3^2) = a_3^2b_2(1-b_2^2) + b_2b_3(b_2^2 - b_3^2)\,\, .$$
Since $b_3 = x_3/x_1$ is not a constant in $k(Z)$, it follows that 
$b_3(1 - b_3^2) \not= 0$ in $k(Z)$, whence also not $0$ in $k(X)$. Thus 
$${\rm (XIV)}\,\, a_2^2 = \frac{a_3^2b_2(1-b_2^2) + b_2b_3(b_2^2 - b_3^2)}
{b_3(1 - b_3^2)}\,\, .$$
Therefore $a_2$ is algebraic over $k(a_3, b_2, b_3)$ of degree at most $2$. 
Since $X$ is of dimension $3$ over $k$, it follows that $a_3, b_2, b_3$ form a 
transcendence basis of $k(X)$ over $k$. Thus, the subring $k[a_3, b_2, b_3]$ 
of $k(X)$ is isomorphic to the polynomial ring over $k$ of Krull-dimension $3$. Moreover, the right hand side of (XIV) 
is not a square in $k(a_3, b_2, b_3)$. Indeed, the multiplicity of $b_3$ in the denominator is $1$ while the numerator is not in $k$ and the multiplicity of $b_3$ in the numerator is $0$. 
Thus the equation (XIV) is the minimal polynomial of $a_2$ over $k(a_3, b_2, b_3)$. Hence $X$ is birationally equivalent to the double cover of ${\AAA}^3 = {\rm Spec}\, k[a_3, b_2, b_3]$, defined by (XIV). This means that $X$ is birationally equivalent to the hypersurface in the affine space ${\AAA}^4 = {\rm Spec}\, k[a, \alpha, b, \beta]$, defined by (XIV) or equivalently defined by (XIII) or by (XII), in which $(a_2, a_3, b_2, b_3)$ are replaced by $(a, \alpha, b, \beta)$. 
\end{proof}

\begin{corollary}\label{conic}
Let $H \subset {\AAA}^4 = {\rm Spec}\, k[a, \alpha, b, \beta]$ be the same as in Proposition (\ref{hypersurface}). Consider the affine plane ${\AAA}^2 = {\rm Spec}\, k[b, \beta]$ and the natural projection 
$$\pi : {\AAA}^4 \rightarrow {\AAA}^2$$ 
defined by 
$$(a,b, \alpha, \beta) \mapsto (b, \beta)\,\, .$$ 
Then the natural restriction map
$$p := \pi\vert H : H \rightarrow {\AAA}^2$$ 
is a conic bundle over ${\AAA}^2$. 
In particular, the graph $\Gamma$ of the rational map $\tilde{p} : X \cdots\rightarrow {\PP}^2$ naturally induced by $p$ forms a conic bundle on $\Gamma$ over ${\PP}^2$. We note that $\Gamma$ is projective and birationally equivalent to 
$X$ over $k$. 
\end{corollary} 

\begin{proof} The fibre  $\pi^{-1}(\eta)$ of $\pi$ over the generic point $\eta \in {\AAA}^2 = {\rm Spec}\, k[b, \beta]$ is the affine space ${\AAA}^2_{\eta} = {\rm Spec}\, k(b, \beta)[a, \alpha]$ defined over $\kappa(\eta) = k(b, \beta)$. Thus by the second equation in Proposition (\ref{hypersurface}), the generic fibre 
$X_{\eta} := (\pi \vert H)^{-1}(\eta)$ is the conic in ${\AAA}^2_{\eta}$, defined by
$$a^2\beta(1-\beta^2) = \alpha^2b(1-b^2) +b\beta(b^2 -\beta^2)\,\, .$$
This implies the result.
\end{proof}

\begin{remark}\label{rationalpoint}
The conic $X_{\eta}$ in the proof of Proposition (\ref{conic}) has no rational point over $\kappa(\eta) = k(b, \beta)$, i.e., the set $X_{\eta}(k(b, \beta))$ is empty. 
\end{remark}

\begin{proof} Suppose to the contrary that $(a(b, \beta), \alpha(b, \beta)) \in X_{\eta}(k(b, \beta))$. We can write 
$$a(b, \beta) = \frac{P(b, \beta)}{Q(b, \beta)}\,\, ,\,\, \alpha(b, \beta) = \frac{R(b, \beta)}{Q(b, \beta)}\,\, ,$$
where $P(b, \beta), Q(b, \beta), R(b, \beta) \in k[b, \beta]$ with no non-constant common factor, possibly after replacing the denominators by their product. Then substituting the above into the equation of $X_{\eta}$ and clearing the denominator, we would have the following 
identity in $k[b, \beta]$:
$$P(b, \beta)^2\beta(1-\beta^2) = R(b, \beta)^2b(1-b^2) + Q(b, \beta)^2b\beta(b^2 - \beta^2)\,\, .$$
Since $k[b, \beta]$ is a polynomial ring, in particular, it is a UFD, it would follow that $P(b, \beta)$ is divisible by $b$ and $R(b, \beta)$ is divisible by $\beta$ in $k[b, \beta]$. Thus $P(b, \beta) = P_1(b, \beta)b$ and $R(b, \beta) = R_1(b, \beta)\beta$ for some $P_1(b, \beta), R_1(b, \beta) \in k[b, \beta]$. Substituting these two into the equality above and dividing by $b\beta \not= 0$, it follows that
$$P_1(b, \beta)^2b(1 -\beta^2) = R_1(b, \beta)^2\beta(1-b^2) + Q(b, \beta)^2(b^2 - \beta^2)\,\, .$$ 

Substitute $b=0$ into this equation: we obtain $R_1(0,\beta )^2\beta + Q(0,\beta )^2\beta ^2=0$, which implies that $R_1(0,\beta )=Q(0,\beta )=0$. This means that both $R_1(b,\beta )$ and $Q(b,\beta )$ are divisible by $b$. Similarly, if we substitute $\beta =0$ into the above equation we find that both $P_1(b,\beta )$ and $Q(b,\beta )$ are divisible by $\beta$. Thus we can write 
\begin{eqnarray*}
P_1(b,\beta )=\beta P_2(b,\beta ),~R_1(b,\beta )=bR_2(b,\beta ),~Q(b,\beta )=b\beta Q_2(b,\beta ),
\end{eqnarray*}
where $P_2(b,\beta ),~R_2(b,\beta ),~Q(b,\beta )\in k[b,\beta ]$. But this implies that all $P(b,\beta ),Q(b,\beta ),R(b,\beta )$ are divisible by $b\beta $, a contradiction. 
\end{proof}

The next corollary completes the proof of Theorem (\ref{main1}):

\begin{corollary}\label{unirational}
Let $H \subset {\AAA}^4 = {\rm Spec}\, k[a, \alpha, b, \beta]$, $p : H \rightarrow {\AAA}^2 = {\rm Spec}\, k[b, \beta]$ be the same as in Proposition \ref{hypersurface} and 
Corollary \ref{conic}. Consider another affine space ${\rm Spec}\, k[s, t]$ 
and the (finite Galois) morphism of degree $4$
$$f : {\rm Spec}\, k[s, t] \rightarrow {\rm Spec}\, k[b, \beta]$$
defined by 
$$f^*b = s^2\,\, ,\,\, f^*\beta = t^2\,\, .$$
Consider then the fibre product 
$$Q := H \times_{{\rm Spec}\, k[b, \beta]} {\rm Spec}\, k[s, t]$$ 
and the natural second projection $p_2 : Q \rightarrow {\rm Spec}\, k[s, t]$. 
Then $p_2$ is a conic bundle with a rational section and $Q$ is a rational threefold. In particular, $H$, hence $X$, is unirational. 
\end{corollary}

\begin{proof} Recall that $H$ is the hypersurface in ${\rm Spec}\, k[a, b, \alpha, \beta]$ defined by
$$a^2\beta(1-\beta^2) = \alpha^2b(1-b^2) +b\beta(b^2 -\beta^2)\,\, ,$$
or equivalently by
$$(a^2 -b)(\alpha^2 - \beta^3) = (\alpha^2 - \beta)(a^2 - b^3)\,\, .$$
Thus, by  definition of the fibre product, $Q$ is a hypersurface 
in the affine space ${\AAA}^4 = {\rm Spec}\, k[a, \alpha, s, t]$, 
defined by
$$a^2t^2(1-t^4) = \alpha^2s^2(1-s^4) +s^2t^2(s^4 -t^4)\,\, ,$$
or equivalently by
$$(a^2 -s^2)(\alpha^2 - t^6) = (\alpha^2 - t^2)(a^2 - s^6)\,\, .$$
Then the natural projection $p_2 : Q \rightarrow {\rm Spec}\, k[s, t]$ is a conic bundle with generic fibre
$$Q_{\eta'} = (a^2t^2(1-t^4) = \alpha^2s^2(1-s^2) +s^2t^2(s^4 -t^4)) \subset {\rm Spec}\, k(s, t)[a, \alpha] = {\AAA}^2_{\eta'}\,\, ,$$
where $\eta'$ is the generic point of ${\rm Spec}\, k[s, t]$. 
Then $Q_{\eta'}$ has a rational point $(a, \alpha ) = (s, t) \in Q(k(s, t))$ over $\kappa(\eta') = k(s, t)$. Hence $Q_{\eta'}$ is isomorphic to ${\PP}^1_{\eta'}$ over $k(s, t)$. Thus, denoting the affine coordinate of ${\PP}^1_{\eta'}$ by $v$, we obtain that 
$$k(Q) = k(s, t)(Q_{\eta'}) \simeq k(s, t)({\PP}^1_{\eta'}) = k(s, t)(v) = k(s, t, v)\,\, .$$ Since $Q$ is of dimension $3$ over $k$, it follows that $s$, $t$, $v$ are algebraically independent over $k$. Hence, $k(Q)$ is isomorphic to the rational function field of ${\PP}^3$ over $k$. Hence $Q$ is 
a rational threefold over $k$, i.e., birationally equivalent to ${\PP}^3$ over 
$k$. Since the natural morphism $p_1 : Q \rightarrow H$, i.e., the first projection morphism in the fibre product, is a finite dominant morphism of degree $4$, $Q$ is birational to ${\PP}^3$ and $H$ is birationally equivalent to $X$, all over $k$, we obtain a rational dominant map $q : {\PP}^3 \cdots\rightarrow X$ over $k$, from 
the natural projection $p_1 : Q \rightarrow H$. Hence $X$ is unirational. 
\end{proof}

\end{document}